\documentclass[11pt]{article}

\usepackage{amssymb, latexsym, mathtools, array, multirow, graphicx, indentfirst, enumerate}
\usepackage[all]{xy}

\usepackage[thmmarks, amsmath]{ntheorem}
{
\theoremstyle{nonumberplain}
\theoremheaderfont{\itshape}
\theorembodyfont{\upshape}
\theoremsymbol{\mbox{$\Box$}}
\newtheorem{proof}{Proof.}
}
\usepackage[hang,flushmargin]{footmisc}
\usepackage{lipsum}
\def\Z{\mathbb Z}
\def\Q{\mathbb Q}
\def\R{\mathbb R}

\theoremstyle{plain}
\newtheorem{lemma}{Lemma}[section]
\newtheorem{proposition}{Proposition}[section]
\newtheorem{theorem}{Theorem}[section]

\newtheorem{claim}{Claim}[section]
\newtheorem*{problem}{Lifting Problem}

\theorembodyfont{\upshape}
\newtheorem{definition}{Definition}[section]

\newtheorem{example}{Example}[section]
\newtheorem{remark}{Remark}[section]

\newcommand\blfootnote[1]{
\begingroup 
\renewcommand\thefootnote{}\footnote{#1}
\addtocounter{footnote}{-1}
\endgroup 
}
\title{Lower bound for Buchstaber invariants of real universal complexes}
\author{Qifan Shen}
\date{}
\begin{document}

\renewcommand{\theequation}{\thesection.\arabic{equation}}
\setcounter{equation}{0} \maketitle

\vspace{-1.0cm}

\bigskip

\noindent {\bf Abstract.} In this article, we prove that Buchstaber invariant of 4-dimensional real universal complex is no less than 24 as a follow-up to the work of Ayzenberg and Sun. Moreover, a lower bound for Buchstaber invariants of $n$-dimensional real universal complexes is given as an improvement of result of Erokhovets.
\blfootnote{
Partially supported by the grant from NSFC (No. 11971112).}

\renewcommand{\theequation}{\thesection.\arabic{equation}}

%Introduction
\section{Introduction} \label{Introduction}
\setcounter{equation}{0}

\emph{Moment-angle complex} and its real counterpart are fundamental objects in toric topology as they construct links among algebraic geometry, sympletic geometry and combinatorics (see Definition \ref{moment-angle complex}). Moreover, they are equipped with certain group actions, yielding applications in both non-equivariant and equivariant categories (see \cite{BP15} for more details). 

For a given simplicial complex $K$ on $m$ vertices, the associated real moment-angle complex $\R\mathcal{Z}_{K}$ (resp. moment-angle complex $\mathcal{Z}_{K}$) admits a natural $\Z_{2}^{m}$-action (resp. $T^{m}$-action) by coordinate-wise sign permutation (resp. rotation). However, these actions fail to be free unless $K$ is the empty complex, leading to the definition of \emph{real Buchstaber invariant} $s_{\R}(K)$ (resp. \emph{Buchstaber invariant} $s(K)$) as the maximal rank of subgroup (resp. toric subgroup) that acts freely on $\R\mathcal{Z}_{K}$ (resp. $\mathcal{Z}_{K}$) (see Definition \ref{Buchstaber invariant}). These two types of invariants measure the degree of symmetry of the corresponding complexes and were first introduced in \cite{BP02} for simplicial spheres with generalization in \cite{FM11} for arbitrary simplicial complexes.

Buchstaber asked for a combinatorial description of $s(K)$ in \cite{BP02}, which turns out to be quite hard and remains open till today. As a matter of fact, calculation of $s(K)$ is not completed even for the special case where $K$ is a dual cyclic polytope and partial results can be found in \cite{Er14}.

On the other hand, there exists a general bound for $s_{\R}(K)$ and $s(K)$:
\begin{equation*}
m-\gamma(K)\leq s(K)\leq s_{\R}(K)\leq m-n 
\end{equation*}
where $K$ is an $(n-1)$-dimensional simplicial complex on $m$ vertices and $\gamma(K)$ stands for ordinary chromatic number of $K$. This formula can be derived from relations among generalized chromatic numbers in a systematic manner \cite{Ay10}. Indeed, with the help of \emph{real universal complex} $\mathcal{K}_{1}^{n}$ and \emph{universal complex} $\mathcal{K}_{2}^{n}$ introduced in \cite{DJ91}, $s_{\R}(K)$ and $s(K)$ can be expressed as $m-r_{\R}(K)$ and $m-r(K)$ respectively, where $r_{\R}(K)$ and $r(K)$ are minimal rank of certain colorings on $K$ (see Section \ref{Preliminaries} or \cite{Iz01}). Moreover, the upper bound of $s(K)$ and $s_{\R}(K)$ is controlled by the sum of rational Betti numbers since $2^{m-n}\leq \sum_{i}\mathrm{dim}H^{i}(\mathcal{Z}_{K};\Q)=\sum_{i}\mathrm{dim}H^{i}(\R\mathcal{Z}_{K};\Q)$ was proved as a special case of Halperin-Carlsson conjecture (see \cite{CL12, Us11}).

The general inequality $s(K)\leq s_{\R}(K)$ follows from the fact that involutions on $T^{m}$ and $\mathcal{Z}_{\mathit{K}}$ induced by complex conjugation have fixed point sets $\Z_{2}^{m}$ and $\R\mathcal{Z}_{\mathit{K}}$ respectively. Thus, any free $T^{r}$-action on $\mathcal{Z}_{\mathit{K}}$ induces a free $\Z_{2}^r$-action on $\R\mathcal{Z}_{\mathit{K}}$. Meanwhile, the special case $s(K)=s_{\R}(K)$ is closely related to \emph{Lifting problem} (see Section \ref{Preliminaries}) presented by L$\ddot{\mathrm{u}}$ at the conference on toric topology held in Osaka in November 2011$\footnote{http://www.sci.osaka-cu.ac.jp/masuda/toric/torictopology2011\_osaka.html}$. Let $\Delta(K)=s_{\R}(K)-s(K)=r(K)-r_{\R}(K)$ denote the difference, then vanishing of $\Delta(K)$ is necessary for the validity of Lifting problem on $K$.

For real universal complex $\mathcal{K}_{1}^{n}$, we have the following monotonicity and universal property:
\begin{proposition}
\rm{\cite[Theorem 3.3]{Sun17}} $\Delta(\mathcal{K}_{1}^{n})\leq \Delta(\mathcal{K}_{1}^{n+1})$.
\end{proposition}
\begin{proposition}
\rm{\cite[Proposition 4.1]{Sun17}} $r_{\R}(K)\leq r_{\R}(\mathcal{K}_{1}^{n}) \Longrightarrow \Delta(K)\leq \Delta(\mathcal{K}_{1}^{n})$.
\end{proposition}
Therefore, the value of $\Delta(\mathcal{K}_{1}^{n})$ is significant as it provides an upper bound for general cases. In \cite{Ay10} and \cite{Ay16}, Ayzenberg showed $\Delta(\mathcal{K}_{1}^{n})=0$ for $n=1,2,3$ and $\Delta(\mathcal{K}_{1}^{4})>0$ respectively. In \cite{Sun17}, $\Delta(\mathcal{K}_{1}^{4})=1$ was confirmed by Sun. Furthermore, an upper bound $\Delta(\mathcal{K}_{1}^{n})\leq 3\cdot2^{n-2}-1-n$ for $n\geq 2$ can be viewed as a corollary of part of the theorem in \cite{Er14} (see Remark \ref{Erokhovets}).

\vspace{0.5cm}
Main aim of this article is to further estimate the upper bound for $\Delta(\mathcal{K}_{1}^{n})$, which can be deduced from construction of certain colorings. Since $s(\mathcal{K}_{1}^{n})+\Delta(\mathcal{K}_{1}^{n})=2^{n}-1-n$ by definition, upper bound estimation of $\Delta(\mathcal{K}_{1}^{n})$ is equivalent to lower bound estimation of $s(\mathcal{K}_{1}^{n})$. Explicit construction of non-degenerate simplicial maps from $\mathcal{K}_{1}^{5}$ to $\mathcal{K}_{2}^{7}$ and from $\mathcal{K}_{1}^{n}$ to $\mathcal{K}_{2}^{2^{n-2}+1}$ for $n\geq 2$ yields the following theorems: 

\begin{theorem} \label{K15 theorem}
$\Delta(\mathcal{K}_{1}^{5})\leq 2$, i.e., $s(\mathcal{K}_{1}^{5})\geq 24$.
\end{theorem}

\begin{theorem} \label{K1n theorem}
For $n\geq 2$, $\Delta(\mathcal{K}_{1}^{n})\leq 2^{n-2}+1-n$, i.e., $s(\mathcal{K}_{1}^{n})\geq 3\cdot2^{n-2}-2$.
\end{theorem}

\begin{remark}
The construction process is equivalent to finding the solution to a system of nonlinear Diophantine equations. These equations all belong to a certain type discussed in \cite{Sa17}, where existence and classification problem of solutions to a single equation was solved. However, existence problem of solutions to the system is much harder to deal with since the number of equations grow rapidly as $n$ increases.
\end{remark}

\begin{remark}
Vanishing of $\Delta(\mathcal{K}_{1}^{n})$ for $n=1,2,3$ follows from the fact that every matrix in $\mathrm{GL}(3,\Z_{2})$ has integral determinant $\pm 1$. For general $n$, upper bound estimation of $\Delta(\mathcal{K}_{1}^{n})$ is also related to determinant calculation in both $\Z_{2}$ and $\Z$. An upper bound $\frac{(n+1)^{\frac{(n+1)}{2}}}{2^{n}}$ for absolute value of integral determinant of matrix in $\mathrm{GL}(n,\Z_{2})$ was given in \cite{Ha93}. This is a special case of \emph{Hadamard maximum determinant problem} which aims to calculate the maximal determinant of a square matrix with elements restricted in a given set $S$. It should be pointed out that this problem is far from being solved even for the simplest case $S=\{0,1\}$, since whether the bound given above is sharp or not remains unknown except for $n+1$ being power of 2. Computational results in low dimensions ($n\leq 9$) with the aid of computer were listed in \cite{Zi06} and recent theoretical progress can be found in \cite{Sh20}.
\end{remark}

This article is organized as follows. In Section \ref{Preliminaries}, basic definitions and notations are listed. Section \ref{K15} is divided into four parts to prove Theorem \ref{K15 theorem} step by step. Section \ref{K1n} includes the proof of Theorem \ref{K1n theorem} with an illustrative example.

%Preliminaries
\section{Preliminaries} \label{Preliminaries}
\setcounter{equation}{0}

In the first place, we shall give the formal definition of (real) moment-angle complex and (real) Buchstaber invariant.

\begin{definition} \label{moment-angle complex}
Given a simplicial complex $K$ on $[m]=\{1,\dots,m\}$, we can define the \emph{real moment-angle complex} $\R\mathcal{Z}_{K}$ and the \emph{moment-angle complex} $\mathcal{Z}_{K}$ associated to $K$:
\begin{align*}
\R\mathcal{Z}_{K}=\bigcup\limits_{I\subset K}(D^{1},S^{0})^{I}\subseteq (D^{1})^{m} \qquad \mathcal{Z}_{K}=\bigcup\limits_{I\subset K}(D^{2},S^{1})^{I}\subseteq (D^{2})^{m}
\end{align*}
where $(X,A)^{I}=\{(x_{1},\dots,x_{m})\in X^{m}, x_{i}\in A\ \emph{if}\ i\notin I\}$ for $A\subseteq X$.
\end{definition}

\begin{definition} \label{Buchstaber invariant}
For $\R\mathcal{Z}_{K}$ and $\mathcal{Z}_{K}$ associated to a simplicial complex $K$ on $[m]$: \\
(1) The \emph{real Buchstaber invariant} $s_{\R}(K)$ is the maximal rank of a subgroup $H\subseteq \Z_{2}^{m}$ such that the restricted action $H\curvearrowright \R\mathcal{Z}_{\mathit{K}}$ is free; \\
(2) The \emph{Buchstaber invariant} $s(K)$ is the maximal rank of a toric subgroup $G\subseteq T^{m}$ such that the restricted action $G\curvearrowright \mathcal{Z}_{\mathit{K}}$ is free.
\end{definition}

\begin{example}
Let $K$ be the boundary of a square with vertices labeled as $1,3,2,4$ counterclockwise. By definition, $\mathcal{Z}_{K}=(D^{2}\times S^{1}\cup S^{1}\times D^{2})\times (D^{2}\times S^{1}\cup S^{1}\times D^{2})=S^{3}\times S^{3}$ while $\R\mathcal{Z}_{K}=(D^{1}\times S^{0}\cup S^{0}\times D^{1})\times (D^{1}\times S^{0}\cup S^{0}\times D^{1})=S^{1}\times S^{1}$. Moreover, $s(K)=s_{\R}(K)=2$ follows from the fact $\gamma(K)=\mathrm{dim}K+1=2$.
\end{example}

Secondly, we introduce the (real) universal complex and the corresponding coloring to get an equivalent expression of (real) Buchstaber invariant.

\begin{definition} \label{universal complex}
Let $R_{d}^{n}=\Z_{2}^{n}$ when $d=1$ and $R_{d}^{n}=\Z^{n}$ when $d=2$. The simplicial complex $\mathcal{K}_{d}^{n}$ is defined on the set of primitive vectors in $R_{d}^{n}$ as follow:
$$[\boldsymbol{v_{1}},\cdots,\boldsymbol{v_{k}}] \mathrm{\ is\ a\ simplex\ of\ } \mathcal{K}_{d}^{n} \Longleftrightarrow \{\boldsymbol{v_{1}},\cdots,\boldsymbol{v_{k}}\} \mathrm{\ is\ part\ of\ a\ basis\ of\ } R_{d}^{n}.$$
$\mathcal{K}_{1}^{n}$ is called \emph{real universal complex} while $\mathcal{K}_{2}^{n}$ is called \emph{universal complex}.
\end{definition}

\begin{definition} \label{coloring}
An \emph{$R_{d}^{r}$-coloring} on a simpicial complex $K$ is defined as a non-degenerate simplicial map $\lambda: K \rightarrow \mathcal{K}_{d}^{r}$. The non-degenerate condition means $\lambda$ is an isomorphism on each simplex of $K$.
\end{definition}

Let $r_{\R}(K)$ denote the minimum value of $r$ such that there exists an $R_{1}^{r}$-coloring on $K$. Similarly, $r(K)$ represents the minimum value of $r$ such that there exists an $R_{2}^{r}$-coloring on $K$. Then it follows from definition that $r_{\R}(\mathcal{K}_{1}^{n})=r(\mathcal{K}_{2}^{n})=n$. In addition, equivalent expressions $s_{\R}(K)=m-r_{\R}(K)$ and $s(K)=m-r(K)$ were first proved in \cite{Iz01}. Since there are $2^{n}-1$ primitive vectors in $\mathcal{K}_{1}^{n}$, we have $s_{\R}(\mathcal{K}_{1}^{n})=s(\mathcal{K}_{1}^{n})+\Delta(\mathcal{K}_{1}^{n})=2^{n}-1-n$.

With notations above, we can formally state the Lifting problem as follow:
\begin{problem} \label{lifting problem}
\rm{\cite[Remark 6]{Lv17}} \it{For} any given simplicial complex $K$ and non-degenerate simplicial map $f: K \rightarrow \mathcal{K}_{1}^{s_{\R}(K)}$, does there exist a lifting map $\widetilde{f}: K \rightarrow \mathcal{K}_{2}^{s_{\R}(K)}$ such that the diagram below is commutative: 
$$\xymatrix{
& \mathcal{K}_{2}^{s_{\R}(K)} \ar[d]^{\pi} \\
K \ar[ru]^{\widetilde{f}} \ar[r]^{f} & \mathcal{K}_{1}^{s_{\R}(K)}
}$$
where $\pi: \mathcal{K}_{2}^{s_{\R}(K)} \rightarrow \mathcal{K}_{1}^{s_{\R}(K)}$ is natural modulo 2 projection.
\end{problem}

%Lower bound for $s(\mathcal{K}_{1}^{5})$
\section{Lower bound for $s(\mathcal{K}_{1}^{5})$} \label{K15}
\setcounter{equation}{0}

%Preparation
\subsection{Preparation}
Let $e$ represent the identity of $\Z_{2}$ and 1 represent the identity of $\Z$ to avoid confusion. As listed in the table below, we can take a partition $vt(\mathcal{K}_{1}^{5})=V_{1}\sqcup V_{2}\sqcup V_{3}\sqcup V_{4}\sqcup V_{5}$ such that $V_{i}$ consists of primitive vectors with $5-i$ zeros and label 31 elements of $vt(\mathcal{K}_{1}^{5})$ in lexicographic order.
\begin{center}
\begin{tabular}{|c|c|c|c|}
\hline
$V_{1}$ & $V_{2}$ & $V_{3}$ & $V_{4}$ \\ \hline
$\boldsymbol{v_{1}}=(e,0,0,0,0)$ & $\boldsymbol{v_{6}}=(e,e,0,0,0)$ & $\boldsymbol{v_{16}}=(e,e,e,0,0)$ & $\boldsymbol{v_{26}}=(e,e,e,e,0)$ \\
$\boldsymbol{v_{2}}=(0,e,0,0,0)$ & $\boldsymbol{v_{7}}=(e,0,e,0,0)$ & $\boldsymbol{v_{17}}=(e,e,0,e,0)$ & $\boldsymbol{v_{27}}=(e,e,e,0,e)$ \\
$\boldsymbol{v_{3}}=(0,0,e,0,0)$ & $\boldsymbol{v_{8}}=(e,0,0,e,0)$ & $\boldsymbol{v_{18}}=(e,e,0,0,e)$ & $\boldsymbol{v_{28}}=(e,e,0,e,e)$ \\
$\boldsymbol{v_{4}}=(0,0,0,e,0)$ & $\boldsymbol{v_{9}}=(e,0,0,0,e)$ & $\boldsymbol{v_{19}}=(e,0,e,e,0)$ & $\boldsymbol{v_{29}}=(e,0,e,e,e)$ \\
$\boldsymbol{v_{5}}=(0,0,0,0,e)$ & $\boldsymbol{v_{10}}=(0,e,e,0,0)$ & $\boldsymbol{v_{20}}=(e,0,e,0,e)$ & $\boldsymbol{v_{30}}=(0,e,e,e,e)$ \\
& $\boldsymbol{v_{11}}=(0,e,0,e,0)$ & $\boldsymbol{v_{21}}=(e,0,0,e,e)$ & \\
& $\boldsymbol{v_{12}}=(0,e,0,0,e)$ & $\boldsymbol{v_{22}}=(0,e,e,e,0)$ & \\
& $\boldsymbol{v_{13}}=(0,0,e,e,0)$ & $\boldsymbol{v_{23}}=(0,e,e,0,e)$ & \\
& $\boldsymbol{v_{14}}=(0,0,e,0,e)$ & $\boldsymbol{v_{24}}=(0,e,0,e,e)$ & \\
& $\boldsymbol{v_{15}}=(0,0,0,e,e)$ & $\boldsymbol{v_{25}}=(0,0,e,e,e)$ & \\ 
\hline
\multicolumn{4}{|c|}{$V_{5}: \boldsymbol{v_{31}}=(e,e,e,e,e)$} \\
\hline
\end{tabular}
\end{center}

Within the rest of this article, we assume all vectors are understood as column vectors and $\mathrm{det}_{\Z_{2}}$(-) represents determinant taken in $\Z_{2}$ while $\mathrm{det}$(-) represents determinant taken in $\Z$. The statement of Theorem 1 is equivalent to $r(\mathcal{K}_{1}^{5})\leq 7$, i.e., there exists a vertex map $\Lambda: vt(\mathcal{K}_{1}^{5}) \rightarrow vt(\mathcal{K}_{2}^{7})$ which induces a non-degenerate simplicial map from $\mathcal{K}_{1}^{5}$ to $\mathcal{K}_{2}^{7}$. By the restriction of non-degenerate condition, it remains to verify 
\begin{align*}
& \forall\ \{\boldsymbol{v_{i_{1}}},\boldsymbol{v_{i_{2}}},\boldsymbol{v_{i_{3}}},\boldsymbol{v_{i_{4}}},\boldsymbol{v_{i_{5}}}\} \subseteq vt(\mathcal{K}_{1}^{5})\ \emph{satisfying}\ \mathrm{det}_{\Z_{2}}(\boldsymbol{v_{i_{1}}},\boldsymbol{v_{i_{2}}},\boldsymbol{v_{i_{3}}},\boldsymbol{v_{i_{4}}},\boldsymbol{v_{i_{5}}})=e, \\
& \exists\ \boldsymbol{\alpha}=(a_{1},\dots,a_{7})^{\mathrm{T}} \in \Z^{7}\ \emph{and}\ \boldsymbol{\beta}=(b_{1},\dots,b_{7})^{\mathrm{T}} \in \Z^{7}\ \emph{with the property}: \tag{$\bigstar$} \\
& \mathrm{det}(\Lambda(\boldsymbol{v_{i_{1}}}),\Lambda(\boldsymbol{v_{i_{2}}}),\Lambda(\boldsymbol{v_{i_{3}}}),\Lambda(\boldsymbol{v_{i_{4}}}),\Lambda(\boldsymbol{v_{i_{5}}}),\boldsymbol{\alpha},\boldsymbol{\beta})=\pm 1.
\end{align*}

Indeed, it is even possible to construct $\Lambda$ with additional restrictions:
\begin{align*}
p_{j}\circ \Lambda=id_{j} \qquad 1 \leq j \leq 5
\end{align*}
where $p_{j}$ is the projection onto the $j^{th}$ coordinate of $\Z^{7}$ and $id_{j}$ is the identity map of the $j^{th}$ coordinate. For the $6^{th}$ and $7^{th}$ coordinate, write $\phi=p_{6}\circ \Lambda$, $\psi=p_{7}\circ \Lambda$ and suppose for each $i\in \{1,\dots,31\}$, $\phi(\boldsymbol{v_{i}})=s_{i}\in \Z$, $\psi(\boldsymbol{v_{i}})=t_{i}\in \Z$. Then write $\Phi(\boldsymbol{v_{i}})=\left(\begin{smallmatrix} \boldsymbol{v_{i}} \\ s_{i} \end{smallmatrix} \right) \in \Z^{6}$, $\Psi(\boldsymbol{v_{i}})=\left(\begin{smallmatrix} \boldsymbol{v_{i}} \\ t_{i} \end{smallmatrix} \right) \in \Z^{6}$ and $\boldsymbol{\alpha'}=(a_{1},\dots,a_{5},a_{6})^{\mathrm{T}}$, $\boldsymbol{\beta'}=(b_{1},\dots,b_{5},b_{7})^{\mathrm{T}}$. Furthermore, let $A=(\boldsymbol{v_{i_{1}}},\boldsymbol{v_{i_{2}}},\boldsymbol{v_{i_{3}}},\boldsymbol{v_{i_{4}}},\boldsymbol{v_{i_{5}}})$ and $(A)_{j}$ denote the matrix $A$ with $j^{th}$ row replaced by $(s_{i_{1}},s_{i_{2}},s_{i_{3}},s_{i_{4}},s_{i_{5}})$,  $(A)^{j}$ denote the matrix $A$ with $j^{th}$ row replaced by $(t_{i_{1}},t_{i_{2}},t_{i_{3}},t_{i_{4}},t_{i_{5}})$. Then by basic linear algebra:
\begin{align*}
& \mathrm{det}(\Phi(\boldsymbol{v_{i_{1}}}),\Phi(\boldsymbol{v_{i_{2}}}),\Phi(\boldsymbol{v_{i_{3}}}),\Phi(\boldsymbol{v_{i_{4}}}),\Phi(\boldsymbol{v_{i_{5}}}),\boldsymbol{\alpha'}) \\
=\ & a_{6}\mathrm{det}A-a_{1}\mathrm{det}(A)_{1}-a_{2}\mathrm{det}(A)_{2}-a_{3}\mathrm{det}(A)_{3}-a_{4}\mathrm{det}(A)_{4}-a_{5}\mathrm{det}(A)_{5}; \\
& \mathrm{det}(\Psi(\boldsymbol{v_{i_{1}}}),\Psi(\boldsymbol{v_{i_{2}}}),\Psi(\boldsymbol{v_{i_{3}}}),\Psi(\boldsymbol{v_{i_{4}}}),\Psi(\boldsymbol{v_{i_{5}}}),\boldsymbol{\beta'}) \\
=\ & b_{7}\mathrm{det}A-b_{1}\mathrm{det}(A)^{1}-b_{2}\mathrm{det}(A)^{2}-b_{3}\mathrm{det}(A)^{3}-b_{4}\mathrm{det}(A)^{4}-b_{5}\mathrm{det}(A)^{5}.
\end{align*}
By Chinese Remainder Theorem, the existence of $\boldsymbol{\alpha'}$ for the first determinant being $\pm 1$ is equivalent to:
\begin{align*}
g.c.d.(\mathrm{det}A,\mathrm{det}(A)_{1},\mathrm{det}(A)_{2},\mathrm{det}(A)_{3},\mathrm{det}(A)_{4},\mathrm{det}(A)_{5})=1. \tag{$\star 1$}
\end{align*}
Similarly, the existence of $\boldsymbol{\beta'}$ for the second determinant being $\pm 1$ is equivalent to 
\begin{align*}
g.c.d.(\mathrm{det}A,\mathrm{det}(A)^{1},\mathrm{det}(A)^{2},\mathrm{det}(A)^{3},\mathrm{det}(A)^{4},\mathrm{det}(A)^{5})=1. \tag{$\star 2$}
\end{align*}
If $(\star 1)$ or $(\star 2)$ holds, then taking $\boldsymbol{\beta}=\boldsymbol{e_{7}}$ or $\boldsymbol{\alpha}=\boldsymbol{e_{6}}$ as standard basis of $\Z^{7}$ yields the validity of $(\bigstar)$. Specifically, there is nothing to verify when $|\mathrm{det}A|=1$ itself. On the other hand, it follows from the upper bound given in \cite{Ha93} that $|\mathrm{det}A|\leq 3$ if $A\in \mathrm{GL}(4,\Z_{2})$ and $|\mathrm{det}A|\leq 5$ if $A\in \mathrm{GL}(5,\Z_{2})$. 

In order to discuss 5-dimensional case, two lemmas in 4-dimensional case are needed.

\begin{lemma} \label{rank=4 det=3}
\rm{\cite[Lemma 3.2]{Sun17}} \it{For} $M\in \mathrm{GL}(4,\Z_{2})$ regarded as an integral matrix, $\mathrm{det}M=\pm 3$ induces: 
$$
M=\begin{pmatrix} 1 & 1 & 1 & 0 \\ 1 & 1 & 0 & 1 \\ 1 & 0 & 1 & 1 \\ 0 & 1 & 1 & 1 \end{pmatrix} 
or \begin{pmatrix} 1 & 1 & 1 & 0 \\ 1 & 0 & 0 & 1 \\ 0 & 1 & 0 & 1 \\ 0 & 0 & 1 & 1 \end{pmatrix}
$$
after necessary permutation of rows and columns.
\end{lemma}

\begin{lemma} \label{rank=4 det=2}
For $M\in Mat(4,\Z_{2})$ regarded as an integral matrix, if $\mathrm{det}M=\pm 2$, then $M$ equals to one of the following types after necessary permutation of rows and columns: 
\begin{align*}
& (1) \quad \begin{pmatrix} 1 & x_{1} & x_{2} & x_{3} \\ 0 & 1 & 1 & 0 \\ 0 & 1 & 0 & 1 \\ 0 & 0 & 1 & 1 \end{pmatrix} \\
& (2) \quad \begin{pmatrix} 1 & 1 & 1 & 1 \\ 1 & 1 & 0 & 0 \\ 1 & 0 & 1 & 0 \\ 1 & 0 & 0 & 1 \end{pmatrix}  \quad
\begin{pmatrix} 1 & 1 & 1 & 0 \\ 1 & 1 & 0 & 1 \\ 1 & 0 & 1 & 1 \\ 1 & 0 & 0 & 0 \end{pmatrix} \\
& (3) \quad \begin{pmatrix} 1 & 1 & 0 & 0 \\ 1 & 0 & 1 & 0 \\ 0 & 1 & 1 & 1 \\ 0 & 0 & 0 & 1 \end{pmatrix}  \quad
\begin{pmatrix} 1 & 0 & y_{1} & y_{2} \\ 1 & 1 & 1 & 0 \\ 1 & 1 & 0 & 1 \\ 0 & 0 & 1 & 1 \end{pmatrix}
\end{align*}
where $x_{1},x_{2},x_{3}$ and $y_{1},y_{2}$ belong to $\{0,1\}$.
\end{lemma}

\begin{proof}
Similar to 5-dimensional case, we can take a partition $vt(\mathcal{K}_{1}^{4})=W_{1}\sqcup W_{2}\sqcup W_{3}\sqcup W_{4}$ such that $W_{i}$ consists of primitive vectors with $4-i$ zeros and label 15 elements of $vt(\mathcal{K}_{1}^{4})$ in lexicographic order. Define two matrices as \emph{equivalent} if they differ from each other by permutation of rows and columns. We analyze the columns of $M=(\boldsymbol{\lambda_{1}},\boldsymbol{\lambda_{2}},\boldsymbol{\lambda_{3}},\boldsymbol{\lambda_{4}})$ in the sequel: 

\vspace{0.5cm}
\noindent $\emph{Case 1.}$ There exists $\boldsymbol{\lambda_{i}} \in W_{1}$, then we can assume $\boldsymbol{\lambda_{1}}=(1,0,0,0)^{\mathrm{T}}$ by equivalence. In this way, $M$ belongs to type (1) since $\left(\begin{smallmatrix} 1 & 1 & 0\\ 1 & 0 & 1\\ 0 & 1 & 1 \end{smallmatrix}\right)$ is the only $3\times 3$ binary matrix with absolute value of determinant equal to 2 up to equivalence. 

\vspace{0.5cm}
\noindent $\emph{Case 2.}$ There exists $\boldsymbol{\lambda_{i}} \in W_{4}$, then we can assume $\boldsymbol{\lambda_{1}}=(1,1,1,1)^{\mathrm{T}}$ by equivalence. Note that in this case the other column vectors can not belong to $W_{1}$ since that is contradictory to $\emph{Case 1}$. Moreover, if $\boldsymbol{\lambda_{2}} \in W_{3}$, then substract $\boldsymbol{\lambda_{2}}$ from $\boldsymbol{\lambda_{1}}$ will lead to contradiction by the same reason. Therefore, all three other vectors belong to $W_{2}$ and type (2) is obtained after taking equivalence. 

\vspace{0.5cm}
\noindent $\emph{Case 3.}$ According to Lemma 1, one column vector belongs to $W_{2}$ and others belong to $W_{2}\sqcup W_{3}$ in the remaining case. We call $\boldsymbol{\lambda_{i}}$ is $\emph{contained}$ in $\boldsymbol{\lambda_{j}}$ if $\lambda_{ik}\leq \lambda_{jk}$ is valid for every $k$. If this happens, then we can assume $\boldsymbol{\lambda_{1}}=(1,1,1,0)^{\mathrm{T}}$ and $\boldsymbol{\lambda_{2}}=(0,1,1,0)^{\mathrm{T}}$ up to equivalence and direct computation gives out the second matrix in type (3). Otherwise, every column vectors must belong to $W_{2}$, which leads to the first matrix of type (3). 
\end{proof}

%Elements in $\mathrm{GL}(5,\Z_{2})$ with integral determinant $\pm 5$
\subsection{Elements in $\mathrm{GL}(5,\Z_{2})$ with integral determinant $\pm 5$}
\begin{claim} \label{rank=5 det=5 A}
$\forall\ M=(\boldsymbol{\lambda_{1}},\boldsymbol{\lambda_{2}},\boldsymbol{\lambda_{3}},\boldsymbol{\lambda_{4}},\boldsymbol{\lambda_{5}})\in \mathrm{GL}(5,\Z_{2})$ with integral determinant $\pm 5$, there is a $4\times 4$ minor $M_{ij}$ corresponding to element $m_{ij}=1$ such that $\mathrm{det}M_{ij}=\pm 3$.
\end{claim}

\begin{proof}
If there exists one row or one column with more than two zeros, then expansion by minors on that row or column leads to a minor corrsponding to element 1 with determinant $\pm 3$. Otherwise, after certain permutation of columns, it can be assumed that $\boldsymbol{\lambda_{5}}\in V_{3}\sqcup V_{4}$ is the column with most zeros. If $\boldsymbol{\lambda_{5}}\in V_{4}$, then $M$ consists of all five elements in $V_{4}$ or any four elements in $V_{4}$ plus $\boldsymbol{v_{31}}$, neither of which has determinant $\pm 5$. Now suppose $\boldsymbol{\lambda_{5}}\in V_{3}$ and all minors of $M$ corresponding to element 1 do not equal to $\pm 3$, then expansion on $\boldsymbol{\lambda_{5}}$ shows that the minors corresponding to element 1 in $\boldsymbol{\lambda_{5}}$ are $(1,2,2)$ up to equivalence. Combine Lemma \ref{rank=4 det=2} and restrictions above, the only possible choice for $M$ is: 
$$\begin{pmatrix}
1 & 0 & y_{1} & y_{2} & l_{1} \\
1 & 1 & 1 & 0 & l_{2} \\
1 & 1 & 0 & 1 & l_{3} \\
0 & 0 & 1 & 1 & 1 \\
n_{1} & 1 & n_{2} & n_{3} & 1
\end{pmatrix}$$
with $y_{1},y_{2},l_{1},l_{2},l_{3},n_{1},n_{2},n_{3}\in \{0,1\}$. \\
Note that if $n_{1}=1$, then subtract $\boldsymbol{\lambda_{2}}$ from $\boldsymbol{\lambda_{1}}$ leads to contradiction. In addition, if $l_{1}=0$, then the sum of row 2,3 and 4 is (2,2,2,2,2), leading to contradiction again. Therefore, $M$ must have the following type:
$$
\begin{pmatrix}
1 & 0 & y_{1} & y_{2} & 1 \\
1 & 1 & 1 & 0 & 0 \\
1 & 1 & 0 & 1 & 0 \\
0 & 0 & 1 & 1 & 1 \\
0 & 1 & n_{2} & n_{3} & 1
\end{pmatrix}
$$
with $y_{1},y_{2},n_{2},n_{3}\in \{0,1\}$. \\
Subtracting $\boldsymbol{\lambda_{2}}$ from $\boldsymbol{\lambda_{1}}$ shows either minor $M_{11}$ or minor $M_{51}$ equals to $\pm 3$, resulting in contradiction anyway. 
\end{proof}

With the help of Claim \ref{rank=5 det=5 A} above, 5-dimensional binary matrices with integral determinant $\pm 5$ can be constructed by expanding from 4-dimensional binary matrices with integral determinant $\pm 3$. As a matter of fact, there are only three equivalent classes in total.

\begin{claim} \label{rank=5 det=5 B}
If $M=(\boldsymbol{\lambda_{1}},\boldsymbol{\lambda_{2}},\boldsymbol{\lambda_{3}},\boldsymbol{\lambda_{4}},\boldsymbol{\lambda_{5}})\in \mathrm{GL}(5,\Z_{2})$ has integral determinant $\pm 5$, then $M$ is equivalent to one of the following three matrices:
\begin{equation*}
M_{1}=\begin{pmatrix} 1 & 1 & 1 & 0 & 0 \\ 1 & 1 & 0 & 1 & 0 \\ 1 & 0 & 1 & 1 & 0 \\ 0 & 1 & 1 & 1 & 1 \\ 1 & 0 & 0 & 0 & 1 \end{pmatrix} \qquad
M_{2}=\begin{pmatrix} 1 & 1 & 1 & 0 & 0 \\ 1 & 0 & 0 & 1 & 0 \\ 0 & 1 & 0 & 1 & 0 \\ 0 & 0 & 1 & 1 & 1 \\ 1 & 1 & 0 & 0 & 1 \end{pmatrix} \qquad
M_{3}=\begin{pmatrix} 1 & 1 & 1 & 0 & 0 \\ 1 & 1 & 0 & 1 & 0 \\ 1 & 0 & 1 & 1 & 1 \\ 0 & 1 & 1 & 1 & 1 \\ 1 & 1 & 0 & 0 & 1 \end{pmatrix}. 
\end{equation*}
\end{claim}

\begin{proof}
Since elements in $\mathrm{GL}(4,\Z_{2})$ can not have integral determinant $\pm 5$, none of the column belongs to $V_{1}$. It turns out for the same reason, none of the column belongs to $V_{5}$ either. If not, fix $\boldsymbol{\lambda_{1}}=\boldsymbol{v_{31}}$, then the other columns do not belong to $V_{4}$ and any two column vectors $\boldsymbol{\lambda_{i}},\boldsymbol{\lambda_{j}}$ belonging to $V_{3}$ share one common zero coordinate, since otherwise basic column operations will lead to a 5-dimensional binary matrix with one column belonging to $V_{1}$. The remaining cases are equivalent to either $\boldsymbol{\lambda_{2}}\in V_{2}$ or $\{\boldsymbol{\lambda_{2}},\boldsymbol{\lambda_{3}},\boldsymbol{\lambda_{4}},\boldsymbol{\lambda_{5}}\}\subseteq V_{3}$ pairwise sharing a common zero coordinate. In the former case, expansion by minors on $\boldsymbol{\lambda_{2}}$ yields contradiction while in the latter case, there exists one row with four zeros, leading to contradition as well. Since the matrix with all five columns in $V_{4}$ has determinant $\pm 4$, we can assume that $\boldsymbol{\lambda_{5}}\in V_{2}\sqcup V_{3}$ is the column with most zeros.

If $\boldsymbol{\lambda_{5}}\in V_{2}$, then combining Lemma \ref{rank=4 det=3} and Claim \ref{rank=5 det=5 A}, $M$ is equivalent to:
\begin{equation*}
\begin{pmatrix} 1 & 1 & 1 & 0 & 0 \\ 1 & 1 & 0 & 1 & 0 \\ 1 & 0 & 1 & 1 & 0 \\ 0 & 1 & 1 & 1 & 1 \\ n_{1} & n_{2} & n_{3} & n_{4} & 1 \end{pmatrix}\ or\ \begin{pmatrix} 1 & 1 & 1 & 0 & 1 \\ 1 & 0 & 0 & 1 & 0 \\ 0 & 1 & 0 & 1 & 0 \\ 0 & 0 & 1 & 1 & 0 \\ n_{1}' & n_{2}' & n_{3}' & n_{4}' & 1 \end{pmatrix}\ or\ \begin{pmatrix} 1 & 1 & 1 & 0 & 0 \\ 1 & 0 & 0 & 1 & 0 \\ 0 & 1 & 0 & 1 & 0 \\ 0 & 0 & 1 & 1 & 1 \\ n_{1}'' & n_{2}'' & n_{3}'' & n_{4}'' & 1 \end{pmatrix} 
\end{equation*}
where $\{n_{i},n_{i}',n_{i}''\}_{i=1}^{4}$ belong to $\{0,1\}$. Direct computation of determinant yields $M_{1}$ for the first matrix and $M_{2}$ for the third matrix while there is no solution for the second matrix above.

Similarly, if $\boldsymbol{\lambda_{5}}\in V_{3}$, then $M$ is equivalent to:
\begin{equation*}
\begin{pmatrix} 1 & 1 & 1 & 0 & 0 \\ 1 & 1 & 0 & 1 & 0 \\ 1 & 0 & 1 & 1 & 1 \\ 0 & 1 & 1 & 1 & 1 \\ n_{1} & n_{2} & n_{3} & n_{4} & 1 \end{pmatrix}\ or\ \begin{pmatrix} 1 & 1 & 1 & 0 & 1 \\ 1 & 0 & 0 & 1 & 1 \\ 0 & 1 & 0 & 1 & 0 \\ 0 & 0 & 1 & 1 & 0 \\ 1 & 1 & 1 & n_{4}' & 1 \end{pmatrix}\ or\ \begin{pmatrix} 1 & 1 & 1 & 0 & 0 \\ 1 & 0 & 0 & 1 & 1 \\ 0 & 1 & 0 & 1 & 1 \\ 0 & 0 & 1 & 1 & 0 \\ 1 & 1 & 1 & n_{4}'' & 1 \end{pmatrix} 
\end{equation*}
with $n_{1},n_{2},n_{3},n_{4},n_{4}',n_{4}''\in \{0,1\}$. Direct computation of determinant yields $M_{3}$ for the first matrix while there is no solution for the second and third matrix above. 
\end{proof}

%Elements in $\mathrm{GL}(5,\Z_{2})$ with integral determinant $\pm 3$
\subsection{Elements in $\mathrm{GL}(5,\Z_{2})$ with integral determinant $\pm 3$}

In this subsection, suppose $\mathrm{det}M=\pm 3$ with $\boldsymbol{\lambda_{5}}$ having most zeros and $m_{55}=1$ by equivalence. According to the proof of Claim \ref{rank=5 det=5 B}, it is necessary for $\boldsymbol{\lambda_{5}}\in V_{1}\sqcup V_{2}\sqcup V_{3}$ in this case. The following discussion is based on expansion by minors on $\boldsymbol{\lambda_{5}}$ with argument similar to Lemma \ref{rank=4 det=2}. 

\vspace{0.5cm}
\noindent \emph{Class 1}\quad $\boldsymbol{\lambda_{5}}\in V_{1}$. \\
By assumption and Lemma \ref{rank=4 det=3}, $\boldsymbol{\lambda_{5}}=\boldsymbol{v_{5}}$ and $M$ is equivalent to one of the following types:
\begin{equation*}
N_{1}^{a}=\begin{pmatrix} 1 & 1 & 1 & 0 & 0 \\ 1 & 1 & 0 & 1 & 0 \\ 1 & 0 & 1 & 1 & 0 \\ 0 & 1 & 1 & 1 & 0 \\ n_{1} & n_{2} & n_{3} & n_{4} & 1 \end{pmatrix}\qquad N_{1}^{b}=\begin{pmatrix} 1 & 1 & 1 & 0 & 0 \\ 1 & 0 & 0 & 1 & 0 \\ 0 & 1 & 0 & 1 & 0 \\ 0 & 0 & 1 & 1 & 0 \\ n_{1}' & n_{2}' & n_{3}' & n_{4}' & 1 \end{pmatrix}
\end{equation*}
where $\{n_{i},n_{i}'\}_{i=1}^{4}$ belong to $\{0,1\}$. 

\vspace{0.5cm}
\noindent \emph{Class 2}\quad $\boldsymbol{\lambda_{5}}\in V_{2}$. \\
The corresponding minors of element 1 in $\boldsymbol{\lambda_{5}}$ are either $\pm (0,3)$ or $\pm (1,2)$ and it can be assumed that $M_{55}=\pm 3$ or $\pm 2$ by equivalence. 

In the former case, computation followed from Lemma \ref{rank=4 det=3} shows that $M$ is equivalent to one of the following types:
\begin{equation*}
N_{2}^{a}=\begin{pmatrix} 1 & 1 & 1 & 0 & 1 \\ 1 & 1 & 0 & 1 & 0 \\ 1 & 0 & 1 & 1 & 0 \\ 0 & 1 & 1 & 1 & 0 \\ n_{1} & n_{2} & n_{3} & n_{4} & 1 \end{pmatrix}\quad N_{2}^{b}=\begin{pmatrix} 1 & 1 & 1 & 0 & 1 \\ 1 & 0 & 0 & 1 & 0 \\ 0 & 1 & 0 & 1 & 0 \\ 0 & 0 & 1 & 1 & 0 \\ n_{1}' & n_{2}' & n_{3}' & n_{4}' & 1 \end{pmatrix}\quad N_{2}^{c}=\begin{pmatrix} 1 & 1 & 1 & 0 & 0 \\ 1 & 0 & 0 & 1 & 0 \\ 0 & 1 & 0 & 1 & 0 \\ 0 & 0 & 1 & 1 & 1 \\ n_{1}'' & n_{2}'' & n_{3}'' & n_{4}'' & 1 \end{pmatrix} 
\end{equation*}
where $\{n_{i},n_{i}',n_{i}''\}_{i=1}^{4}$ belong to $\{0,1\}$ and satisfy $n_{1}+2n_{4}=n_{2}+n_{3},\ n_{1}'+n_{2}'+n_{3}'=n_{4}'$ and $n_{1}''+n_{2}''=2n_{3}''+n_{4}''$ respectively. 

In the latter case, type (1) matrix in Lemma \ref{rank=4 det=2} induces four types of equivalent classes for $M$: \\
\begin{equation*}
N_{3}^{a}=\begin{pmatrix} 1 & 1 & 1 & 1 & 0 \\ 0 & 1 & 1 & 0 & 0 \\ 0 & 1 & 0 & 1 & 0 \\ 0 & 0 & 1 & 1 & 1 \\ 1 & 0 & 0 & 0 & 1 \end{pmatrix} \qquad N_{3}^{b}=\begin{pmatrix} 1 & 0 & x_{2} & x_{3} & 0 \\ 0 & 1 & 1 & 0 & 0\\ 0 & 1 & 0 & 1 & 0 \\ 0 & 0 & 1 & 1 & 1 \\ 1 & 1 & n_{3} & n_{4} & 1 \end{pmatrix}
\end{equation*}
\begin{equation*}
N_{3}^{c}=\begin{pmatrix} 1 & 0 & x_{2}' & x_{3}' & 0 \\ 0 & 1 & 1 & 0 & 0 \\ 0 & 1 & 0 & 1 & 0 \\ 0 & 0 & 1 & 1 & 1 \\ 1 & 0 & n_{3}' & n_{4}' & 1 \end{pmatrix}\qquad N_{3}^{d}=\begin{pmatrix} 1 & 1 & x_{2}'' & x_{3}'' & 0 \\ 0 & 1 & 1 & 0 & 0 \\ 0 & 1 & 0 & 1 & 0 \\ 0 & 0 & 1 & 1 & 1 \\ 1 & 1 & n_{3}'' & n_{4}'' & 1  \end{pmatrix} 
\end{equation*}
where $\{x_{i},x_{i}',x_{i}''\}_{i=2}^{3}$ and $\{n_{j},n_{j}',n_{j}''\}_{j=3}^{4}$ belong to $\{0,1\}$ and satisfy $x_{2}+x_{3}=n_{3}+n_{4},\ x_{2}'+x_{3}'=n_{3}'+n_{4}'+1$ and $x_{2}''+x_{3}''=n_{3}''+n_{4}''+1$ respectively. \\
Similarly, the first matrix of type (2) in Lemma \ref{rank=4 det=2} induces two types of equivalent classes for $M$: \\
\begin{equation*}
N_{4}^{a}=\begin{pmatrix} 1 & 1 & 1 & 1 & 1 \\ 1 & 1 & 0 & 0 & 0 \\ 1 & 0 & 1 & 0 & 0 \\ 1 & 0 & 0 & 1 & 0 \\ 1 & 0 & 0 & 0 & 1 \end{pmatrix} \qquad N_{4}^{b}=\begin{pmatrix} 1 & 1 & 1 & 1 & 0 \\ 1 & 1 & 0 & 0 & 0\\ 1 & 0 & 1 & 0 & 0 \\ 1 & 0 & 0 & 1 & 1 \\ n_{1} & n_{2} & n_{3} & n_{4} & 1 \end{pmatrix}
\end{equation*}
where $\{n_{i}\}_{i=1}^{4}$ belong to $\{0,1\}$ and satisfy $n_{1}+n_{4}+1=n_{2}+n_{3}$. \\
While the second matrix of type (2) in Lemma \ref{rank=4 det=2} induces three types of equivalent classes for $M$: \\
\begin{equation*}
N_{5}^{a}=\begin{pmatrix} 1 & 1 & 1 & 0 & 1 \\ 1 & 1 & 0 & 1 & 0 \\ 1 & 0 & 1 & 1 & 0 \\ 1 & 0 & 0 & 0 & 0 \\ n_{1} & 0 & 0 & 0 & 1 \end{pmatrix} \quad N_{5}^{b}=\begin{pmatrix} 1 & 1 & 1 & 0 & 0 \\ 1 & 1 & 0 & 1 & 0\\ 1 & 0 & 1 & 1 & 0 \\ 1 & 0 & 0 & 0 & 1 \\ 0 & 1 & 0 & 0 & 1 \end{pmatrix} \quad N_{5}^{c}=\begin{pmatrix} 1 & 1 & 1 & 0 & 0 \\ 1 & 1 & 0 & 1 & 0\\ 1 & 0 & 1 & 1 & 0 \\ 1 & 0 & 0 & 0 & 1 \\ 1 & 1 & 1 & 1 & 1 \end{pmatrix}
\end{equation*}
with $n_{1}\in \{0,1\}$. \\
Similarly, the first matrix of type (3) in Lemma \ref{rank=4 det=2} induces two types of equivalent classes for $M$: \\
\begin{equation*}
N_{6}^{a}=\begin{pmatrix} 1 & 1 & 0 & 0 & 0 \\ 1 & 0 & 1 & 0 & 0 \\ 0 & 1 & 1 & 1 & 1 \\ 0 & 0 & 0 & 1 & 0 \\ 1 & 0 & 0 & n_{4} & 1 \end{pmatrix} \qquad N_{6}^{b}=\begin{pmatrix} 1 & 1 & 0 & 0 & 0 \\ 1 & 0 & 1 & 0 & 0\\ 0 & 1 & 1 & 1 & 0 \\ 0 & 0 & 0 & 1 & 1 \\ n_{1}' & n_{2}' & n_{3}' & 0 & 1 \end{pmatrix}
\end{equation*}
where $\{n_{1}',n_{2}',n_{3}',n_{4}\}$ belong to $\{0,1\}$ and satisfy $n_{1}'+1=n_{2}'+n_{3}'$. \\
And the second matrix of type (3) in Lemma \ref{rank=4 det=2} induces two types of equivalent classes for $M$: \\
\begin{equation*}
N_{7}^{a}=\begin{pmatrix} 1 & 0 & y_{1} & y_{2} & 0 \\ 1 & 1 & 1 & 0 & 0 \\ 1 & 1 & 0 & 1 & 0 \\ 0 & 0 & 1 & 1 & 1 \\ n_{1} & n_{2} & n_{3} & n_{4} & 1 \end{pmatrix} \qquad N_{7}^{b}=\begin{pmatrix} 1 & 0 & y_{1}' & y_{2}' & 0 \\ 1 & 1 & 1 & 0 & 0\\ 1 & 1 & 0 & 1 & 1 \\ 0 & 0 & 1 & 1 & 0 \\ n_{1}' & n_{2}' & n_{3}' & n_{4}' & 1 \end{pmatrix}
\end{equation*}
where $\{n_{i},n_{i}'\}_{i=1}^{4}$ and $\{y_{j},y_{j}'\}_{j=1}^{2}$ belong to $\{0,1\}$ and satisfy $n_{2}-n_{3}-n_{4}+(y_{1}+y_{2})(n_{1}-n_{2})=1$ and $n_{3}'-n_{2}'-n_{4}'+(y_{2}'-y_{1}')(n_{1}'-n_{2}')=1$ respectively. 

\vspace{0.5cm}
\noindent \emph{Class 3}\quad $\boldsymbol{\lambda_{5}}\in V_{3}$. \\
This class can be further divided into three cases with regard to the largest absolute value of minors corresponding to element 1 in $\boldsymbol{\lambda_{5}}$. 

If the minor equals to $\pm 3$, then there are two types of equivalent classes induced by matrices in Lemma \ref{rank=4 det=3}: 
\begin{equation*}
N_{8}^{a}=\begin{pmatrix} 1 & 1 & 1 & 0 & 0 \\ 1 & 1 & 0 & 1 & 0 \\ 1 & 0 & 1 & 1 & 1 \\ 0 & 1 & 1 & 1 & 1 \\ n_{1} & n_{2} & n_{3} & n_{4} & 1 \end{pmatrix} \qquad N_{8}^{b}=\begin{pmatrix} 1 & 1 & 1 & 0 & 0 \\ 1 & 0 & 0 & 1 & 0\\ 0 & 1 & 0 & 1 & 1 \\ 0 & 0 & 1 & 1 & 1 \\ 1 & 1 & 1 & 0 & 1 \end{pmatrix}
\end{equation*}
where $\{n_{i}\}_{i=1}^{4}$ belong to $\{0,1\}$ and satisfy $n_{1}+n_{2}=2n_{3}+2n_{4}$. 

For the other two cases, argument can be simplified by the following two claims:
\begin{claim} \label{rank=5 det=3 A}
If $M\in \mathrm{GL}(5,\Z_{2})$ has integral determinant $\pm 3$ and $\boldsymbol{\lambda_{5}}\in V_{3}$ as the column with most zeros such that there exists a minor corresponding to element 1 in $\boldsymbol{\lambda_{5}}$ equal to $\pm 2$, then submatrix corresponding to this minor must be the second matrix of type \rm{(3)} in \rm{Lemma \ref{rank=4 det=2}}.
\end{claim}
\begin{proof}
Without loss of generality, we can suppose $\mathrm{det}M_{55}=\pm 2$, then a straightforward check on matrices listed in Lemma \ref{rank=4 det=2} verifies this claim.

For type (1), $\boldsymbol{\lambda_{1}}$ has more zeros than $\boldsymbol{\lambda_{5}}$, which is not allowed. For type (2), the sum of first row to fourth row is (4,2,2,2,2), leading to a contradiction against $\mathrm{det}(M)=\pm 3$. The first matrix in type (3) can not appear for the same reason. 
\end{proof}

\begin{claim} \label{rank=5 det=3 B}
If $M\in \mathrm{GL}(5,\Z_{2})$ has integral determinant $\pm 3$ and $\boldsymbol{\lambda_{5}}\in V_{3}$ as the column with most zeros such that minors corresponding to element 1 in $\boldsymbol{\lambda_{5}}$ are $\pm (1,1,1)$, then either $M$ consists of four columns in $V_{3}$ plus $\boldsymbol{v_{31}}$ or $M$ has a column in $V_{3}$ contained in another column in $V_{4}$.
\end{claim}

\begin{proof}
Since coexistence of an element in $V_{4}$ and $\boldsymbol{v_{31}}$ is not allowed by determinant restriction, it suffices to discuss cases where $M$ is consisted of elements all belong to $V_{3}$ or $V_{3}\sqcup V_{4}$.

In the former case, every element 1 in $M$ must have minor $\pm 1$ and there exist two columns with a common zero coordinate. Suppose $\boldsymbol{\lambda_{4}}=(0,1,1,1,0)^{\mathrm{T}}, \boldsymbol{\lambda_{5}}=(0,1,1,0,1)^{\mathrm{T}}$ by equivalence, then the first row must be $(1,1,1,0,0)$, which in turn guarantees the existence of $\boldsymbol{v_{21}}$. Suppose $\boldsymbol{\lambda_{3}}=\boldsymbol{v_{21}}=(1,0,0,1,1)^{\mathrm{T}}$ by equivalence again, then $M$ must be equivalent to
\begin{equation*}
\begin{pmatrix} 1 & 1 & 1 & 0 & 0 \\ 1 & 0 & 0 & 1 & 1 \\ 0 & 1 & 0 & 1 & 1 \\ 1 & 0 & 1 & 1 & 0 \\ 0 & 1 & 1 & 0 & 1 \end{pmatrix}. 
\end{equation*}
However, the $(1,1)$-element has minor equal to 2, resulting in contradiction.

In the latter case, if the number of columns in $V_{3}$ is no more than two, then there must be one column in $V_{3}$ contained in another column belonging to $V_{4}$. If there are three columns belonging to $V_{3}$ and none of them is contained in another column belonging to $V_{4}$, then they must share two common nonzero coordinates i.e., $M$ is equivalent to
\begin{equation*}
\begin{pmatrix} 1 & 1 & 1 & 0 & 0 \\ 1 & 1 & 0 & 1 & 0 \\ 1 & 1 & 0 & 0 & 1 \\ 1 & 0 & 1 & 1 & 1 \\ 0 & 1 & 1 & 1 & 1 \end{pmatrix}. 
\end{equation*}
However, the $(3,5)$-element has minor equal to $-$3, resulting in contradiction. If four columns of $M$ is in $V_{3}$ and none of them is contained in the rest column belonging to $V_{4}$, then we can suppose $\boldsymbol{\lambda_{1}}=(0,1,1,1,1)^{\mathrm{T}}$ by equivalence. Expansion by minors on the first row leads to even determinant, which is not allowed. 
\end{proof}

Combining Claim \ref{rank=5 det=3 A} and \ref{rank=5 det=3 B} with Lemma \ref{rank=4 det=3}, the remaining possible equivalent classes for $M$ can be listed below:
\begin{equation*}
N_{9}^{a}=\begin{pmatrix} 1 & 0 & 0 & 0 & 0 \\ 1 & 1 & 1 & 1 & 0 \\ 1 & 1 & 1 & 0 & 1 \\ 1 & 1 & 0 & 1 & 1 \\ 1 & 0 & 1 & 1 & 1 \end{pmatrix} \qquad N_{9}^{b}=\begin{pmatrix} n_{1} & n_{2} & n_{3} & 1 & 0 \\ 1 & 1 & 1 & 0 & 0\\ 1 & 1 & 0 & 1 & 1 \\ 1 & 0 & 1 & 1 & 1 \\ 0 & 1 & 1 & 1 & 1 \end{pmatrix}
\end{equation*}
with $\{n_{i}\}_{i=1}^{3}$ belonging to $\{0,1\}$.

\begin{remark}
Equivalent classes in this subsection may overlap. For instance, taking $n_{1}'=n_{2}'=n_{3}'=1$ in $N_{6}^{b}$ leads to the same equivalent class as $N_{3}^{a}$. In fact, there are 51 equivalent classes according to the counting table in \cite{Zi06}.
\end{remark}

%Construction of $\phi$ and $\psi$
\subsection{Construction of $\phi$ and $\psi$}

Now we are in the position to construct map $\phi$ and $\psi$ for the validity of $(\star 1)$ or $(\star 2)$. Restrict $\phi$ and $\psi$ to be constant on each $V_{i}$, then it suffices to do verification at the level of equivalent class.
\begin{claim} \label{coloring value}
$\phi(\boldsymbol{v_{i}})=\left\{\begin{aligned} 0\quad & \boldsymbol{v_{i}}\in V_{1} \\ 1\quad & \boldsymbol{v_{i}}\in V_{2}\sqcup V_{3}\sqcup V_{4} \\ 2\quad & \boldsymbol{v_{i}}\in V_{5} \end{aligned} \right.$ and $\psi(\boldsymbol{v_{i}})=\left\{\begin{aligned} 0\quad & \boldsymbol{v_{i}}\in V_{2}\sqcup V_{4} \\ 1\quad & \boldsymbol{v_{i}}\in V_{3} \\ 2\quad & \boldsymbol{v_{i}}\in V_{1}\sqcup V_{5} \end{aligned} \right.$ guarantee the validity of $(\star 1)$ or $(\star 2)$ in all equivalent classes above.
\end{claim}

\begin{proof}
The proof follows from straightforward calculation.

%修改=和格式

For $\{M_{i}\}_{i=1}^{3}$ in Claim \ref{rank=4 det=2}, $\mathrm{det}(M_{1})_{5}=-2$, $\mathrm{det}(M_{2})_{5}=2$ and $\mathrm{det}(M_{3})_{5}=-1$, all coprime to $\pm 5$, leading to the validity of $(\star 1)$.

Similarly, $\mathrm{det}(N_{1}^{a})_{1}=-1$ and $\mathrm{det}(N_{1}^{b})_{1}=2$; $\mathrm{det}(N_{2}^{a})_{5}=-2$, $\mathrm{det}(N_{2}^{b})_{5}=1$ and $\mathrm{det}(N_{2}^{c})_{5}=2$; $\mathrm{det}(N_{4}^{a})_{5}=-1$ and $\mathrm{det}(N_{4}^{b})_{5}=-2$ or $-1$; $\mathrm{det}(N_{6}^{a})_{5}=-1$ and $\mathrm{det}(N_{6}^{b})_{5}=-1$; $\mathrm{det}(N_{7}^{a})_{5}=-1$ and $\mathrm{det}(N_{7}^{b})_{5}=-1$; $\mathrm{det}(N_{8}^{a})_{5}=-1$ and $\mathrm{det}(N_{8}^{b})_{5}=1$; $\mathrm{det}(N_{9}^{a})_{2}=-1$ and $\mathrm{det}(N_{9}^{b})_{5}=1$. All these determinants are coprime to $\pm 3$, making $(\star 1)$ valid in these classes.

The remaining parts are of $N_{3}^{*}$-type and $N_{5}^{*}$-type, in which $\psi$ is needed. In fact, problem lies in $N_{3}^{b}, N_{3}^{c}$ and $N_{5}^{b}$ since $\mathrm{det}(N_{3}^{a})_{5}=-2$ and $\mathrm{det}(N_{3}^{d})_{5}=-1$ or $-2$; $\mathrm{det}(N_{5}^{a})_{5}=1$ and $\mathrm{det}(N_{5}^{c})_{5}=1$. For $N_{3}^{b}$, if $n_{3}+n_{4}=0$, then $\mathrm{det}(N_{3}^{b})_{5}=-1$ and if $n_{3}+n_{4}=1$, then $\mathrm{det}(N_{3}^{b})_{5}=-2$. However, $\mathrm{det}(N_{3}^{b})_{j}\equiv 0 \pmod{3}$ for $1 \leq j \leq 5$ if $n_{3}+n_{4}=2$. Similarly, if $n_{3}'+n_{4}'=0$, then $\mathrm{det}(N_{3}^{c})_{5}=-2$ while $\mathrm{det}(N_{3}^{c})_{j}\equiv 0 \pmod{3}$ for $1 \leq j \leq 5$ if $n_{3}'+n_{4}'=1$. For $N_{5}^{b}$, each $\mathrm{det}(N_{5}^{b})_{j}$ is also divisible by 3. To sum up, matrices that can not satisfy $(\star 1)$ with map $\phi$ defined above are:
\begin{equation*}
X_{1}=\begin{pmatrix} 1 & 0 & 1 & 1 & 0 \\ 0 & 1 & 1 & 0 & 0 \\ 0 & 1 & 0 & 1 & 0 \\ 0 & 0 & 1 & 1 & 1 \\ 1 & 1 & 1 & 1 & 1 \end{pmatrix}\quad X_{2}=\begin{pmatrix} 1 & 0 & 1 & 1 & 0 \\ 0 & 1 & 1 & 0 & 0 \\ 0 & 1 & 0 & 1 & 0 \\ 0 & 0 & 1 & 1 & 1 \\ 1 & 0 & 1 & 0 & 1 \end{pmatrix}\quad X_{3}=\begin{pmatrix} 1 & 1 & 1 & 0 & 0 \\ 1 & 1 & 0 & 1 & 0 \\ 1 & 0 & 1 & 1 & 0 \\ 1 & 0 & 0 & 0 & 1 \\ 0 & 1 & 0 & 0 & 1 \end{pmatrix}. 
\end{equation*}
Note that $X_{2}$ can be converted into $X_{3}$ by row permutation $\sigma=\left(\begin{smallmatrix} 1 & 2 & 3 & 4 & 5 \\ 1 & 4 & 5 & 2 & 3 \end{smallmatrix}\right)$ composed with column permutation $\tau=\left(\begin{smallmatrix} 1 & 2 & 3 & 4 & 5 \\ 3 & 5 & 1 & 2 & 4 \end{smallmatrix}\right)$, i.e., they belong to the same equivalent class. Therefore, calculation results $\mathrm{det}(X_{1})^{5}=-1$ and $\mathrm{det}(X_{2})^{5}=1$ complete the proof. 
\end{proof}

\begin{remark}
The value of $\psi$ on set $V_{1}\sqcup V_{5}$ is irrelevant since matrices belonging to $N_{3}^{*}$-type and $N_{5}^{*}$-type do not include column vectors in $V_{1}\sqcup V_{5}$.
\end{remark}

\begin{remark}
The value of $\phi$ can be taken modulo 15 since only coprimeness to 3 and 5 is concerned. Similarly, the value of $\psi$ can be taken modulo 3.
\end{remark}

%Lower bound for $s(\mathcal{K}_{1}^{n})$
\section{Lower bound for $s(\mathcal{K}_{1}^{n})$} \label{K1n}
\setcounter{equation}{0}

By Proposition 2, $\Delta(\mathcal{K}_{1}^{n})$ can be viewed as an upper bound for general cases. However, it remains open whether or not $\Delta(\mathcal{K}_{1}^{n})$ is bounded when n goes to $+\infty$. On the other hand, by mapping $2^{n}-1$ primitive vectors in $vt(\mathcal{K}_{1}^{n})$ to standard basis of $\Z^{2^{n}-1}$, one can easily verify that $\Delta(\mathcal{K}_{1}^{n})\leq 2^{n}-1-n$. With some symmetric modifications, this upper bound can be improved to $2^{n-2}+1-n$ for $n\geq 2$, as stated in Theorem \ref{K1n theorem}.

\begin{proof}
Since $n\geq 2$, one can choose two arbitrary primitive vectors $\boldsymbol{x},\boldsymbol{y}\in vt(\mathcal{K}_{1}^{n})$, then a partition of $vt(\mathcal{K}_{1}^{n})$ is given by $A_{0}=\{\boldsymbol{x},\boldsymbol{y},\boldsymbol{x}+\boldsymbol{y}\}$ and $A_{i}=\{\boldsymbol{a}_{i},\boldsymbol{a}_{i}+\boldsymbol{x},\boldsymbol{a}_{i}+\boldsymbol{y},\boldsymbol{a}_{i}+\boldsymbol{x}+\boldsymbol{y}\}$ for $i=1,\dots,2^{n-2}-1$ with addition taken in $\Z_{2}$. Define a vertex map $\Lambda: vt(\mathcal{K}_{1}^{n})\rightarrow vt(\mathcal{K}_{2}^{2^{n-2}+1})$ with the following assignment: 
\begin{align*}
& \boldsymbol{x}\mapsto \boldsymbol{e_{1}} \qquad \boldsymbol{y}\mapsto \boldsymbol{e_{2}} \qquad \boldsymbol{a_{i}}\mapsto \boldsymbol{e_{i+2}} \\
& \boldsymbol{x+y}\mapsto \boldsymbol{e_{1}+e_{2}} \\
& \boldsymbol{a_{i}+x}\mapsto \boldsymbol{e_{1}+e_{i+2}} \\
& \boldsymbol{a_{i}+y}\mapsto \boldsymbol{e_{2}+e_{i+2}} \\
& \boldsymbol{a_{i}+x+y}\mapsto \boldsymbol{e_{1}+e_{2}+e_{i+2}}
\end{align*}
where $\{\boldsymbol{e_{j}}\}_{j=1}^{2^{n-2}+1}$ is the standard basis. It suffices to verify that $\Lambda$ induces a non-degenerate simplicial map $\widetilde{\Lambda}$ from $\mathcal{K}_{1}^{n}$ to $\mathcal{K}_{2}^{2^{n-2}+1}$. Apparently, for each simplex $\sigma \in \mathcal{K}_{1}^{n}$, $A_{0}\not\subseteq vt(\sigma)$. Let $\sharp(X)$ denote the number of elements in set $X$, then there are three different cases: 

\vspace{0.5cm}
\noindent \emph{Case 1}\quad $\sharp(A_{0}\cap vt(\sigma))=2$. \\
By linear dependency, $\sharp(A_{i}\cap vt(\sigma))\leq 1$ for any $i\geq 1$. Thus, images of $vt(\sigma)$ are parts of columns in the matrix equivalent to
$$
P=\begin{pmatrix}
P_{2} & * \\
0 & I_{2^{n-2}-1}
\end{pmatrix}
$$
where $P_{2}$=$\left(\begin{smallmatrix} 1 & 0\\ 0 & 1 \end{smallmatrix}\right)$ or $\left(\begin{smallmatrix} 1 & 1\\ 0 & 1 \end{smallmatrix}\right)$ and $I_{2^{n-2}-1}$ stands for identity matrix of dimension $2^{n-2}-1$. 

\vspace{0.5cm}
\noindent \emph{Case 2}\quad $\sharp(A_{0}\cap vt(\sigma))=1$. \\
By linear dependency, there exists at most one index $i_{0}\geq 1$ such that $\sharp(A_{i_{0}}\cap vt(\sigma))=2$ while $\sharp(A_{i}\cap vt(\sigma))\leq 1$ is valid for any other $i\geq 1$. Take subtraction between columns if such $i_{0}$ does exist, then images of $vt(\sigma)$ are parts of columns in the matrix equivalent to
$$
Q=\begin{pmatrix}
Q_{2} & * \\
0 & I_{2^{n-2}-1}
\end{pmatrix}
$$
where $Q_{2}$=$\left(\begin{smallmatrix} 1 & 0\\ 0 & 1 \end{smallmatrix}\right)$ or $\left(\begin{smallmatrix} 1 & 1\\ 0 & 1 \end{smallmatrix}\right)$ or $\left(\begin{smallmatrix} 1 & -1\\ 0 & 1 \end{smallmatrix}\right)$. 

\vspace{0.5cm}
\noindent \emph{Case 3}\quad $\sharp(A_{0}\cap vt(\sigma))=0$. \\
Similar to \emph{Case 2}, either there exists at most one index $j_{0}\geq 1$ such that $\sharp(A_{j_{0}}\cap vt(\sigma))=3$ while $\sharp(A_{j}\cap vt(\sigma))\leq 1$ for any other $j\geq 1$, or there are at most two indices $j_{1},j_{2}\geq 1$ such that $\sharp(A_{j_{1}}\cap vt(\sigma))$=$\sharp(A_{j_{2}}\cap vt(\sigma))=2$ while $\sharp(A_{j}\cap vt(\sigma))\leq 1$ for any other $j\geq 1$. Take subtraction between columns if such $j_{0}$ or $j_{1},j_{2}$ do exist, then images of $vt(\sigma)$ can also be viewed as parts of columns in the matrix equivalent to $Q$. \\
Since $\mathrm{det}(P)=\mathrm{det}(Q)=1$, the induced map $\widetilde{\Lambda}$ is non-degenerate as desired. 
\end{proof}

\begin{remark} \label{Erokhovets}
Erokhovets \cite{Er14} gave an upper bound of $r(K)$ for general simplicial complex $K$ on $[m]$ in terms of minimal non-simplices: $r(K)\leq \sum_{i=0}^{l}{\mathrm{dim}\omega_{i}}$ if there exists a collection of minimal non-simplices $\{\omega_{i}\}_{i=0}^{l}$ such that $\cup_{i=0}^{l}{\omega_{i}}=[m]$. Take $K=\mathcal{K}_{1}^{n}$ and $\omega_{i}=A_{i}$, then an upper bound $3\cdot2^{n-2}-1-n$ is obtained for $\Delta(\mathcal{K}_{1}^{n})$. Theorem \ref{K1n theorem} can be regarded as an improvement of this result and it gives sharp upper bound when $n\leq 4$.
\end{remark}

\begin{remark} 
Choosing one primitive vector in $\mathcal{K}_{1}^{n}$, an upper bound $2^{n-1}-n$ can be obtained for any $n\geq 1$ by similar argument. However, similar construction can not give out better results. If three linearly independent primitive vectors in $\mathcal{K}_{1}^{n}$ are chosen at the beginning, then for any simplex $\sigma\in \mathcal{K}_{1}^{n}$, the images of $vt(\sigma)$ can be viewed as parts of columns in the matrix equivalent to 
$$
R=\begin{pmatrix}
R_{3} & * \\
0 & I_{2^{n-3}-1}
\end{pmatrix}.
$$
Here $R_{3}$ may be equal to $\left(\begin{smallmatrix} 1 & 1 & 1\\ 1 & 0 & -1\\ 1 & -1 & 0\end{smallmatrix}\right)$ due to necessary column subtractions, leading to $\mathrm{det}R=-3$ instead of $\pm 1$. Starting from choosing more linearly independent primitive vectors causes more problems like this.
\end{remark}

\begin{example}
For $n=4$, take primitive vectors $\boldsymbol{x},\boldsymbol{y}$ as $(1,0,0,0)^{\mathrm{T}}$ and $(0,1,0,0)^{\mathrm{T}}$ respectively, then $\Lambda: vt(\mathcal{K}_{1}^{4})\rightarrow vt(\mathcal{K}_{2}^{5})$ is defined as follow: 
\begin{align*}
& \begin{pmatrix}
1 & 0 & 1 & 0 & 1 & 0 & 1 & 0 & 1 & 0 & 1 & 0 & 1 & 0 & 1 \\
0 & 1 & 1 & 0 & 0 & 1 & 1 & 0 & 0 & 1 & 1 & 0 & 0 & 1 & 1 \\
0 & 0 & 0 & 1 & 1 & 1 & 1 & 0 & 0 & 0 & 0 & 1 & 1 & 1 & 1 \\
0 & 0 & 0 & 0 & 0 & 0 & 0 & 1 & 1 & 1 & 1 & 1 & 1 & 1 & 1
\end{pmatrix} \\
\mapsto & \begin{pmatrix}
1 & 0 & 1 & 0 & 1 & 0 & 1 & 0 & 1 & 0 & 1 & 0 & 1 & 0 & 1 \\
0 & 1 & 1 & 0 & 0 & 1 & 1 & 0 & 0 & 1 & 1 & 0 & 0 & 1 & 1 \\
0 & 0 & 0 & 1 & 1 & 1 & 1 & 0 & 0 & 0 & 0 & 0 & 0 & 0 & 0 \\
0 & 0 & 0 & 0 & 0 & 0 & 0 & 1 & 1 & 1 & 1 & 0 & 0 & 0 & 0 \\
0 & 0 & 0 & 0 & 0 & 0 & 0 & 0 & 0 & 0 & 0 & 1 & 1 & 1 & 1
\end{pmatrix}. 
\end{align*}
Since $p_{j}\circ \Lambda\neq id_{j}$ for $j
\leq 4$, this map is different from the construction given in \cite{Sun17}.
\end{example}

%Appendix：用于说明在两个附加限制之下，无法确定$\Delta(\mathcal{K}_{1})^{5}$的取值为1还是2。
\section{Appendix} \label{appendix}

If $\Delta(\mathcal{K}_{1}^{5})=1$, then a map $f: vt(\mathcal{K}_{1}^{5}) \rightarrow vt(\mathcal{K}_{2}^{6})$ can be constructed such that it induces a non-degenerate simplicial map $\tilde{f}: \mathcal{K}_{1}^{5} \rightarrow \mathcal{K}_{2}^{6}$. The following claim shows that such map does not exist if two additional restrictions used in the construction of $\Lambda: vt(\mathcal{K}_{1}^{5}) \rightarrow vt(\mathcal{K}_{2}^{7})$ are required. 

\begin{claim}
There does not exist a map $f: vt(\mathcal{K}_{1}^{5}) \rightarrow vt(\mathcal{K}_{2}^{6})$ satisfying the following conditions: \\
(1) $\forall\ \{\boldsymbol{v_{i_{1}}},\boldsymbol{v_{i_{2}}},\boldsymbol{v_{i_{3}}},\boldsymbol{v_{i_{4}}},\boldsymbol{v_{i_{5}}}\} \subseteq vt(\mathcal{K}_{1}^{5})$ with $\mathrm{det}_{\Z_{2}}(\boldsymbol{v_{i_{1}}},\boldsymbol{v_{i_{2}}},\boldsymbol{v_{i_{3}}},\boldsymbol{v_{i_{4}}},\boldsymbol{v_{i_{5}}})=e$, $\exists\ \boldsymbol{\alpha}=(a_{1},\dots,a_{6})^{\mathrm{T}} \in \Z^{6}$ such that $\mathrm{det}(f(\boldsymbol{v_{i_{1}}}),f(\boldsymbol{v_{i_{2}}}),f(\boldsymbol{v_{i_{3}}}),f(\boldsymbol{v_{i_{4}}}),f(\boldsymbol{v_{i_{5}}}),\boldsymbol{\alpha})=\pm 1$. \\
(2) $p_{j} \circ f = id_{j}$ for $1 \leq j \leq 5$. \\
(3) $f|_{V_{j}}$ is constant for $1 \leq j \leq 5$. 
\end{claim}

\begin{proof}
Write $\phi=p_{6} \circ f$ and $w_{j}=\phi|_{V_{j}}$ for $1 \leq j \leq 5$. The claim follows from the fact that (1)-(3) can not be satisfied simutaneously for matrices \\
\[
N_{1}^{a}(0,0,0,0)=\begin{pmatrix} 1 & 1 & 1 & 0 & 0 \\ 1 & 1 & 0 & 1 & 0 \\ 1 & 0 & 1 & 1 & 0 \\ 0 & 1 & 1 & 1 & 0 \\ 0 & 0 & 0 & 0 & 1 \end{pmatrix} \qquad X_{2}=\begin{pmatrix} 1 & 0 & 1 & 1 & 0 \\ 0 & 1 & 1 & 0 & 0 \\ 0 & 1 & 0 & 1 & 0 \\ 0 & 0 & 1 & 1 & 1 \\ 1 & 0 & 1 & 0 & 1 \end{pmatrix}.
\]

Due to the fact $\mathrm{det}N_{1}^{a}=\mathrm{det}X_{2}=-3$, only modulo 3 values of $\{w_{j}\}_{j=1}^{5}$ need to be verified. Since in $X_{2}$, the third row and the fifth row add up to $(1,1,1,1,1)$ and sum of all rows equals to $(2,2,4,3,2)$, $(w_{2},w_{3},w_{4})$ can not be any linear combination of $(1,1,1)$ and $(2,0,1)$. Similarly, verification on $N_{1}^{a}(0,0,0,0)$ shows that $(w_{2},w_{3},w_{4})=(1,0,0)$ is not allowed. Note that $(1,1,1),(2,0,1)$ and $(1,0,0)$ form a basis of $\mathbb{Z}^{3}$(and thus $\mathbb{Z}_{3}^{3}$), there is no solution for $(w_{2},w_{3},w_{4})$. 
\end{proof}

\bigskip

\noindent \textbf{Acknowledgement.} The author would like to thank professor Zhi L$\ddot{\mathrm{u}}$ for introducing this topic to him and making valuable  discussions.

\mbox{}\\
Qifan Shen\\
School of Mathematical Sciences  \\
Fudan University \\
220 Handan Road  \\
Shanghai 200433  \\
People's Republic of China \\
E-Mail: qfshen17@fudan.edu.cn \mbox{}\\

\end{document}